\documentclass[a4paper,oneside,10pt]{article}%
\usepackage{amsmath}
\usepackage{amsfonts}
\usepackage{amssymb}
\usepackage{graphicx}
\usepackage{color}
\usepackage[square,numbers,sort&compress]{natbib}%
\providecommand{\U}[1]{\protect\rule{.1in}{.1in}}

\newtheorem{theorem}{Theorem}[section]

\newtheorem{definition}[theorem]{Definition}
\newtheorem{assumption}[theorem]{Assumption}
\newtheorem{example}[theorem]{Example}

\newtheorem{lemma}[theorem]{Lemma}

\newtheorem{proposition}[theorem]{Proposition}
\newtheorem{remark}[theorem]{Remark}

\newenvironment{proof}[1][Proof]{\noindent\textbf{#1.} }{\ \rule{0.5em}{0.5em}}
\numberwithin{equation}{section}

\begin{document}

\title{The Neyman-Pearson lemma for convex expectations }
\author{Chuanfeng Sun\thanks{School of Mathematical Sciences, University of Jinan,
Jinan, Shandong 250022, P.R. China. e-mail: sms\_suncf@ujn.edu.cn. This
research is partially supported by the National Natural Science Foundation of
China (No. 11701214), the Natural Science Foundation of Shandong Province (No.
ZR2017BA032).} \quad Shaolin Ji\thanks{Zhongtai Institute of Finance, Shandong
University, Jinan, Shandong 250100, PR China. e-mail: jsl@sdu.edu.cn. This
research is supported by National Natural Science Foundation of China (No.
11571203), the Programme of Introducing Talents of Discipline to Universities
of China (No. B12023).}}
\date{}
\maketitle

\textbf{Abstract}. We study the Neyman-Pearson problem for convex expectations
on $L^{\infty}(\mu)$. The existence of the optimal test is given. Without
assuming that the level sets of penalty functions are weakly compact, we prove
that the optimal tests for convex expectations on $L^{\infty}(\mu)$ are just
the classical Neyman-Pearson tests between a fixed representative pair of
simple hypotheses. Then we show that the Neyman-Pearson problem for convex
expectations on $L^{1}(\mu)$ can be solved similarly.



\section{Introduction}

The classical Neyman-Pearson lemma gives the most powerful test for
discriminating between two probability measures and has important applications
in various fields (see \cite{fe}, \cite{le-ro}).

It is well known that many phenomena need to be explored by nonlinear
probabilities or expectations. In 1954, Choquet \cite{r3} extended the
probability measure to the capacity and gave a nonlinear integral named after
him. The coherent risk measure was proposed by Artzner et al. \cite{r5} and
the $g$-expectation was initiated by Peng \cite{r4} in 1999. F\"{o}llmer and
Schied \cite{r6} generalized the coherent risk measure to the convex risk
measure in 2002. Divergence risk measures were considered by Ben-Tal and
Teboulle \cite{BT} under the name of optimized certainty equivalents.

Along with the development of the above concepts, several nonlinear versions
of Neyman-Pearson lemma have also been established. In 1973, Huber and
Strassen \cite{{r7}} studied the Neyman-Pearson lemma for capacities.
Cvitani\'{c} and Karatzas \cite{r1} extended the classical Neyman-Pearson
theory for testing composite hypotheses versus composite alternatives which
can also be understood as discriminating between two sublinear expectations in
2001. Later Schied \cite{r2} gave a Neyman-Pearson lemma for law-invariant
coherent risk measures and robust utility functionals. Ji and Zhou \cite{r8}
studied hypothesis tests for $g$-probabilities in 2008. Rudloff and Karatzas
\cite{r9} studied composite hypothesis by using convex duality in 2010. Apart
from their own theoretical value, the nonlinear versions of Neyman-Pearson
lemma have been found to have many applications especially in finance. For
instance, F\"{o}llmer and Leukert \cite{fo-le-1999} and \cite{fo-le-2000}
studied the quantile hedging and efficient hedging which minimize the
shortfall risk in an incomplete financial market. Rudloff \cite{r10} found a
self-financing strategy that minimize the convex risk of the shortfall using
convex duality method.

In most literatures, the convex duality method is employed to study the
nonlinear Neyman-Pearson lemma or "Neyman-Pearson type" optimization problem.
Without assuming the set of densities which generate the sublinear expectation
is weakly compact, Cvitani\'{c} and Karatzas \cite{r1} studied the
Neyman-Pearson lemma for sublinear expectations on $L^{\infty}(\mu)$. However,
due to the additional penalty function terms, the convex duality method in
\cite{r1} is difficult to apply to the case for convex expectations on
$L^{\infty}(\mu)$. In order to measure the shortfall risk, F\"{o}llmer and
Leukert \cite{fo-le-2000} adopted a specific convex risk measure and Rudloff
\cite{r10} used the convex risk measure on $L^{1}(\mu)$. In both cases, the
sets of densities which generate the convex risk measure are weakly compact.

As in \cite{r1}, it is natural to study the Neyman-Pearson lemma on
$L^{\infty}(\mu)$. So in this paper, we investigate the Neyman-Pearson lemma
for convex expectations (convex risk measures) on $L^{\infty}(\mu)$ and do not
assume that the set of densities which generate the convex expectation is
weakly compact. In more details, for two given convex expectations $\rho_{1}$,
$\rho_{2}$ and a significance level $\alpha$, we want to find an optimal test
$X^{\ast}$ which minimize the convex expectation of Type II error, among all
tests that keep the convex expectation of Type I error below the given
acceptable significance level $\alpha\in(0,1)$. In other words, we study the
following problem:%
\begin{equation}
\text{minimize}\,\rho_{2}(1-X), \label{prelimilary}%
\end{equation}
over the set $\mathcal{X}_{\alpha}=\{X\in L^{\infty}(\mu):0\leq X\leq1,$
$\rho_{1}(X)\leq\alpha\}$.

Instead of being weakly compact, we only assume that the level sets of penalty
functions are closed under the $\mu$-a.e. convergence. Under this assumption,
we can't directly apply the approach in \cite{r1}. The key to solving this
problem is that we find the feasible set $\mathcal{X}_{\alpha}$ is compact in
the weak$^{\ast}$ topology $\sigma(L^{\infty},L^{1})$. Based on this, we can
apply the minimax theorem and find the representative probability measure
$Q^{\ast}$ for $\rho_{2}$. By solving the dual problem, the representative
probability measure $P^{\ast}$ for $\rho_{1}$ is also found. Thus, the optimal
tests for convex expectations are just the classical Neyman-Pearson tests
between a fixed representative pair $(P^{\ast},Q^{\ast})$.

As a by-product, we found that similar ideas can be used to solve the
Neyman-Pearson problem for convex expectations on $L^{1}(\mu)$. So we put this
result in the appendix and gave a brief proof.

This paper is organized as follows: In Section 2, we give some preliminaries
and formulate the simple hypothesis testing problem for convex expectations on
$L^{\infty}(\mu)$. The existence of the optimal tests is derived in section 3.
In section 4, we obtain the form of the optimal tests. An application is given
to illustrate our main result in section 5. Finally, in the appendix we show
that if convex expectations are continuous from above, then Assumption
\ref{assumption} holds naturally and give the Neyman-Pearson lemma for convex
expectations on $L^{1}(\mu)$.

\section{Preliminaries and Problem Formulation}

Let $(\Omega,\mathcal{F},\mu)$ be a probability space and $\mathcal{M}$ be the
set of probability measures on $(\Omega,\mathcal{F})$ that are absolutely
continuous with respect to $\mu$. $P$ and $Q$ are probability measures and
their Radon-Nikodym derivatives $\frac{dP}{d\mu}$ and $\frac{dQ}{d\mu}$ are
denoted as $G_{P}$ and $H_{Q}$ respectively.

\begin{definition}
A mapping $\rho$: $L^{\infty}(\mu)\rightarrow\mathbb{R}$ is called a convex
expectation if for any $X, Y\in L^{\infty}(\mu)$, we have

(i) Monotonicity: If $X\geq Y$, then $\rho(X)\geq\rho(Y)$;

(ii) Invariance: If $c$ is a constant, then $\rho(X+c)=\rho(X)+c$;

(iii) Convexity: If $\lambda\in[0, 1]$, then $\rho\big(\lambda X+(1-\lambda
)Y\big)\leq\lambda\rho(X)+(1-\lambda)\rho(Y)$.
\end{definition}

If we take $\rho^{\prime}(X):=\rho(-X)$, then $\rho^{\prime}$ is a convex risk measure.

\begin{definition}
\label{continuous from below} We call a convex expectation $\rho$ is
continuous from below iff for any sequence $\{X_{n}\}_{n\geq1}\subset
L^{\infty}(\mu)$ increases to some $X\in L^{\infty}(\mu)$, then $\rho
(X_{n})\rightarrow\rho(X)$.
\end{definition}

The following theorem comes from Theorem 6 and Proposition 7 in \cite{r6}.

\begin{theorem}
\label{representation} If a convex expectation $\rho$ is continuous from
below, then

i) For any $X\in L^{\infty}(\mu)$,
\begin{equation}
\rho(X)=\sup\limits_{P\in\mathcal{M}}\big(E_{P}[X]-\rho^{\ast}(P)\big),
\end{equation}
where $\rho^{\ast}$ is the penalty function of $\rho$ and $\rho^{\ast}%
(P)=\sup\limits_{X\in L^{\infty}(\mu)}\big(E_{P}[X]-\rho(X)\big)$.

ii) For any bounded sequence $\{X_{n}\}_{n\geq1}\subset L^{\infty}(\mu)$, if
$X_{n}$ converges to some $X\in L^{\infty}(\mu)$ in probability, then
$\rho(X)\leq\liminf\limits_{n\rightarrow\infty}\rho(X_{n})$.
\end{theorem}

Given two convex expectations $\rho_{1}$ and $\rho_{2}$, for a significance
level $\alpha$ and two random variables $K_{1}$ and $K_{2}$ belonging to
$L^{\infty}(\mu)$ such that $0\leq K_{1}<K_{2}$, we are interested in the
following problem:
\begin{equation}
\text{minimize}\quad\rho_{2}(K_{2}-X), \label{initial-problem}%
\end{equation}
over the set $\mathcal{X}_{\alpha}=\{X:K_{1}\leq X\leq K_{2},\rho_{1}%
(X)\leq\alpha,X\in L^{\infty}(\mu)\}$. Without loss of generality, we assume
$\rho_{1}(K_{1})\leq\alpha\leq\rho_{1}(K_{2})$. Note that if $K_{1}=0$ and
$K_{2}=1$, then the above problem becomes the problem (\ref{prelimilary}).

We call $X^{\ast}$ is the optimal test of (\ref{initial-problem}) if $X^{\ast
}\in\mathcal{X}_{\alpha}$ and
\begin{equation}
\rho_{2}(K_{2}-X^{\ast})=\inf\limits_{X\in\mathcal{X}_{\alpha}}\rho_{2}%
(K_{2}-X).
\end{equation}

By (i) of Theorem \ref{representation},
\[
\rho_{1}(X)=\sup\limits_{P\in\mathcal{M}}E_{P}[X]-\rho_{1}^{\ast}%
(P)\quad\text{and}\quad\rho_{2}(X)=\sup\limits_{Q\in\mathcal{M}}E_{Q}%
[X]-\rho_{2}^{\ast}(Q).
\]
If we denote
\[
\mathcal{P}=\{P:P\in\mathcal{M},\rho_{1}^{\ast}(P)<\infty\}\quad
\text{and}\quad\mathcal{Q}=\{Q:Q\in\mathcal{M},\rho_{2}^{\ast}(Q)<\infty\},
\]
then $\mathcal{P}$ and $\mathcal{Q}$ are nonempty convex sets and
\[
\rho_{1}(X)=\sup\limits_{P\in\mathcal{P}}E_{P}[X]-\rho_{1}^{\ast}%
(P)\quad\text{and}\quad\rho_{2}(X)=\sup\limits_{Q\in\mathcal{Q}}E_{Q}%
[X]-\rho_{2}^{\ast}(Q).
\]
Thus, the problem (\ref{initial-problem}) can also be considered as
discriminating between two convex expectations $\rho_{1}$ and $\rho_{2}$
generated by $\mathcal{P}$ and $\mathcal{Q}$.

\section{The existence of the optimal test \ }

Set $\beta=\inf\limits_{X\in\mathcal{X}_{\alpha}}\rho_{2}(K_{2}-X)$. The
following result shows that the optimal test exists.

\begin{theorem}
\label{existence} If $\rho_{1}$ and $\rho_{2}$ are convex expectations
continuous from below, then the optimal test of (\ref{initial-problem}) exists.
\end{theorem}

\begin{proof}
Take a sequence $\{X_{n}\}_{n\geq1}\subset\mathcal{X}_{\alpha}$ such that
\[
\rho_{2}(K_{2}-X_{n})<\beta+\frac{1}{2^{n}}.
\]
By the Koml\'{o}s theorem, there exist a subsequence $\{X_{n_{i}}\}_{i\geq1}$
of $\{X_{n}\}_{n\geq1}$ and a random variable $X^{\ast}$ such that \
\begin{equation}
\lim_{k\rightarrow\infty}\frac{1}{k}\sum_{i=1}^{k}X_{n_{i}}=X^{\ast},\quad
\mu-a.e..
\end{equation}
Since for any $n$, $K_{1}\leq X_{n}\leq K_{2}$, we have $K_{1}\leq X^{\ast
}\leq K_{2}$, $\mu$-a.e.. By (ii) of Theorem \ref{representation},
\[
\rho_{1}(X^{\ast})\leq\liminf_{k\rightarrow\infty}\rho_{1}(\frac{1}{k}%
\sum_{i=1}^{k}X_{n_{i}})\leq\liminf_{k\rightarrow\infty}\frac{1}{k}\sum
_{i=1}^{k}\rho_{1}(X_{n_{i}})\leq\alpha
\]
which leads to $X^{\ast}\in\mathcal{X}_{\alpha}$. On the other hand,
\[
\rho_{2}(K_{2}-X^{\ast})\leq\liminf_{k\rightarrow\infty}\frac{1}{k}\sum
_{i=1}^{k}\rho_{2}(K_{2}-X_{n_{i}})\leq\beta+\lim_{k\rightarrow\infty}\frac
{1}{k}=\beta.
\]
Thus,
\[
\rho_{2}(K_{2}-X^{\ast})=\inf\limits_{X\in\mathcal{X}_{\alpha}}\rho_{2}%
(K_{2}-X).
\]
This completes the proof.
\end{proof}

\section{The form of the optimal test}

Note that
\[
\inf\limits_{X\in\mathcal{X}_{\alpha}}\rho_{2}(K_{2}-X)=\inf\limits_{X\in
\mathcal{X}_{\alpha}}\sup\limits_{Q\in\mathcal{Q}}\big(E_{Q}[K_{2}-X]-\rho
_{2}^{\ast}(Q)\big).
\]
Then $X^{\ast}$ is the optimal test of (\ref{initial-problem}) if and only if
it is the optimal test of the problem:
\begin{equation}
\text{minimize}\quad\sup\limits_{Q\in\mathcal{Q}}\big(E_{Q}[K_{2}-X]-\rho
_{2}^{\ast}(Q)\big), \label{extended-problem}%
\end{equation}
over $\mathcal{X}_{\alpha}$.

Now we focus on solving the problem (\ref{extended-problem}). Denote the level
sets of penalty functions $\rho_{1}^{\ast}$ and $\rho_{2}^{\ast}$ as
\[
\mathcal{G}_{c}=\{G_{P}:P\in\mathcal{P},\text{\ }\rho_{1}^{\ast}(P)\leq
c\}\quad\text{and}\quad\mathcal{H}_{c}=\{H_{Q}:Q\in\mathcal{Q},\text{\ }%
\rho_{2}^{\ast}(Q)\leq c\},
\]
where $c$ is a constant. Since $\rho_{1}^{\ast}$ and $\rho_{2}^{\ast}$ are
convex functions on $\mathcal{M}$, then both $\mathcal{G}_{c}$ and
$\mathcal{H}_{c}$ are convex sets.

Since $K_{1}$ and $K_{2}$ belong to $L^{\infty}(\mu)$, we denote the least
upper bound of them by $M$.

\begin{assumption}
\label{assumption} There exist two constants $u>\max\{0, M-\rho_{1}(0)+1\}$
and $v>\max\{0, M-\rho_{2}(0)+1\}$ such that $\mathcal{G}_{u}$ and
$\mathcal{H}_{v}$ are both closed under the $\mu$-a.e. convergence.
\end{assumption}

Since the penalty function of the sublinear expectation takes only the values
$0$ and $+\infty$, for sublinear case, Assumption \ref{assumption} is equal to
require $\{G_{P}:P\in P\}$ and $\{H_{Q}:Q\in Q\}$ are both closed under the
$\mu$-a.e. convergence, which is similar as the assumption given by
Cvitani\'{c} and Karatzas in \cite{r1}. In Appendix, we show that if $\rho
_{1}$ and $\rho_{2}$ are continuous from above, then Assumption
\ref{assumption} holds naturally.

\subsection{The existence of a representative probability $Q^{\ast}$}

In this subsection, we want to find a representative probability $Q^{\ast}$
such that
\[
\inf\limits_{X\in\mathcal{X}_{\alpha}}\sup\limits_{Q\in\mathcal{Q}}%
\big(E_{Q}[K_{2}-X]-\rho_{2}^{\ast}(Q)\big)=\inf\limits_{X\in\mathcal{X}%
_{\alpha}}E_{Q^{\ast}}[K_{2}-X]-\rho_{2}^{\ast}(Q^{\ast})\text{.}%
\]
If such a $Q^{\ast}$ exists, then for any optimal test $X^{\ast}$ of
(\ref{initial-problem}), we have
\[
\sup\limits_{Q\in\mathcal{Q}}\big(E_{Q}[K_{2}-X^{\ast}]-\rho_{2}^{\ast
}(Q)\big)=\inf\limits_{X\in\mathcal{X}_{\alpha}}E_{Q^{\ast}}[K_{2}-X]-\rho
_{2}^{\ast}(Q^{\ast}),
\]
which leads to $E_{Q^{\ast}}[K_{2}-X^{\ast}]=\inf\limits_{X\in\mathcal{X}%
_{\alpha}}E_{Q^{\ast}}[K_{2}-X]$.

\begin{theorem}
\label{minimax-result} If $\rho_{1}$ and $\rho_{2}$ are convex expectations
continuous from below and Assumption \ref{assumption} holds, then there exists
$Q^{\ast}\in\mathcal{Q}$ such that for any optimal test $X^{\ast}$ of
(\ref{initial-problem}), we have
\begin{equation}
E_{Q^{\ast}}[K_{2}-X^{\ast}]=\inf\limits_{X\in\mathcal{X}_{\alpha}}E_{Q^{\ast
}}[K_{2}-X].
\end{equation}

\end{theorem}

Before proving Theorem \ref{minimax-result}, we first give some lemmas.

\begin{lemma}
\label{Fatou-property-2} For any sequence $\{Q_{n}\}_{n\geq1}\subset
\mathcal{M}$, if $H_{Q_{n}}$ converges to some $H_{Q_{0}}$ under $L^{1}(\mu)$
norm, then
\begin{equation}
\inf\limits_{X\in\mathcal{X}_{\alpha}}E_{Q_{0}}[K_{2}-X]\geq\limsup
_{n\rightarrow\infty}\inf\limits_{X\in\mathcal{X}_{\alpha}}E_{Q_{n}}[K_{2}-X].
\end{equation}

\end{lemma}

\begin{proof}
For any $X\in\mathcal{X}_{\alpha}$, we have
\[
E_{Q_{0}}[K_{2}-X]=\lim\limits_{n\to\infty}E_{Q_{n}}[K_{2}-X]\geq
\limsup\limits_{n\to\infty}\inf\limits_{X\in\mathcal{X}_{\alpha}}E_{Q_{n}%
}[K_{2}-X].
\]
Then
\[
\inf\limits_{X\in\mathcal{X}_{\alpha}}E_{Q_{0}}[K_{2}-X]\geq\limsup
\limits_{n\to\infty}\inf\limits_{X\in\mathcal{X}_{\alpha}}E_{Q_{n}}[K_{2}-X].
\]
This completes the proof.
\end{proof}

\begin{lemma}
\label{weak*-compact} If $\rho_{1}$ is a convex expectation continuous from
below, then $\mathcal{X}_{\alpha}$ is compact in the weak$^{*}$ topology
$\sigma(L^{\infty}(\mu), L^{1}(\mu))$.
\end{lemma}

\begin{proof}
Define $\phi(Y)=\sup\limits_{X\in\mathcal{X}_{\alpha}}E_{\mu}[X\cdot Y]$,
where $Y\in L^{1}(\mu)$. Then $\phi$ is a sublinear function on $L^{1}(\mu)$
and dominated by $M||\cdot||_{L^{1}(\mu)}$. Set
\begin{equation}
\hat{\mathcal{X}}_{\alpha}=\{X\in L^{\infty}(\mu):E_{\mu}[X\cdot Y]\leq
\phi(Y)\ \text{for any}\ Y\in L^{1}(\mu)\}.
\end{equation}
By Theorem 4.2 of chapter I in \cite{r11}, $\hat{\mathcal{X}}_{\alpha}$ is
compact in the weak$^{\ast}$ topology $\sigma(L^{\infty}(\mu),L^{1}(\mu))$.
Then we only need to show
\[
\mathcal{X}_{\alpha}=\hat{\mathcal{X}}_{\alpha}.
\]
Since $\mathcal{X}_{\alpha}\subset\hat{\mathcal{X}}_{\alpha}$ obviously, in
the next, we will show $\hat{\mathcal{X}}_{\alpha}\subset\mathcal{X}_{\alpha}$.

Firstly, for any $\hat{X}\in\hat{\mathcal{X}}_{\alpha}$, we show that
$K_{1}\leq\hat{X}\leq K_{2}$, $\mu$-a.e.. If there exists $\hat{X}\in
\hat{\mathcal{X}}_{\alpha}$ such that $\mu(\{\omega:\hat{X}(\omega
)<K_{1}\})\not =0$, then there will exist a constant $\epsilon>0$ such that
$\mu(A)\not =0$, where $A=\{\omega:\hat{X}(\omega)\leq K_{1}-\epsilon\}$. For
any $X\in\mathcal{X}_{\alpha}$, since $\hat{X}\leq K_{1}-\epsilon$ on $A$, we
have $\hat{X}\leq X-\epsilon$ on $A$. Let $h_{A}=-\dfrac{I_{A}}{\mu(A)}$.
Then
\[
E_{\mu}[\hat{X}h_{A}]=-\dfrac{1}{\mu(A)}E_{\mu}[\hat{X}I_{A}]\geq-\dfrac
{1}{\mu(A)}E_{\mu}[(X-\epsilon)I_{A}]=E_{\mu}[Xh_{A}]+\epsilon.
\]
Due to $X$ can be taken in $\mathcal{X}_{\alpha}$ arbitrarily, we have
\[
E_{\mu}[\hat{X}h_{A}]\geq\sup\limits_{X\in\mathcal{X}_{\alpha}}E_{\mu}%
[Xh_{A}]+\epsilon>\sup\limits_{X\in\mathcal{X}_{\alpha}}E_{\mu}[Xh_{A}%
]=\phi(h_{A}).
\]
Since $h_{A}\in L^{1}(\mu)$, it contradicts with $\hat{X}\in\hat{\mathcal{X}%
}_{\alpha}$. Thus, $\hat{X}\geq K_{1}$, $\mu$-a.e.. Similarly, we can prove
$\hat{X}\leq K_{2}$, $\mu$-a.e..

Next, we show for any $\hat{X}\in\hat{\mathcal{X}}_{\alpha}$, $\rho_{1}%
(\hat{X})\leq\alpha$. Since $\hat{X}\in\hat{\mathcal{X}}_{\alpha}$, for any
$P\in\mathcal{P}$,
\[
E_{P}[\hat{X}]=E_{\mu}[\hat{X}G_{P}]\leq\sup\limits_{X\in\mathcal{X}_{\alpha}%
}E_{\mu}[XG_{P}]=\sup\limits_{X\in\mathcal{X}_{\alpha}}E_{P}[X].
\]
Then
\[%
\begin{array}
[c]{r@{}l}%
\rho_{1}(\hat{X})= & \sup\limits_{P\in\mathcal{P}}\big(E_{P}[\hat{X}]-\rho
_{1}^{\ast}(P)\big)\\
\leq & \sup\limits_{P\in\mathcal{P}}\sup\limits_{X\in\mathcal{X}_{\alpha
^{\ast}}}\big(E_{P}[X]-\rho_{1}^{\ast}(P)\big)\\
= & \sup\limits_{X\in\mathcal{X}_{\alpha}}\sup\limits_{P\in\mathcal{P}%
}\big(E_{P}[X]-\rho_{1}^{\ast}(P)\big)\\
= & \sup\limits_{X\in\mathcal{X}_{\alpha}}\rho_{1}(X)\leq\alpha.
\end{array}
\]
Thus, $\hat{X}\in\mathcal{X}_{\alpha}$.
\end{proof}

\begin{remark}
If $\rho_{1}$ degenerates to be a sublinear expectation, the above result can
also be found in \cite{r10}.
\end{remark}

\begin{lemma}
\label{aux-minimax} If $\rho_{1}$ and $\rho_{2}$ are convex expectations
continuous from below, then
\begin{equation}
\inf\limits_{X\in\mathcal{X}_{\alpha}}\sup\limits_{Q\in\mathcal{Q}}%
\big(E_{Q}[K_{2}-X]-\rho_{2}^{\ast}(Q)\big)=\sup\limits_{Q\in\mathcal{Q}}%
\inf\limits_{X\in\mathcal{X}_{\alpha}}\big(E_{Q}[K_{2}-X]-\rho_{2}^{\ast
}(Q)\big).\label{3.3}%
\end{equation}

\end{lemma}

\begin{proof}
Since for each $X\in\mathcal{X}_{\alpha}$, $E_{Q}[K_{2}-X]-\rho_{2}^{\ast}(Q)$
is a concave function on $\mathcal{Q}$ and for each $Q\in\mathcal{Q}$,
$E_{Q}[K_{2}-X]-\rho_{2}^{\ast}(Q)$ is a linear continuous function on
$L^{\infty}(\mu)$, with $\mathcal{X}_{\alpha}$ is compact in the weak$^{\ast}$
topology $\sigma(L^{\infty}(\mu),L^{1}(\mu))$, then by the minimax theorem
(Refer to Theorem 3.2 of chapter I in \cite{r11}), the equation (\ref{3.3}) holds.
\end{proof}

The following lemma shows that $\rho^{\ast}$ is lower semi-continuous.

\begin{lemma}
\label{Fatou-property-1} If $\rho$ is a convex expectation continuous from
below, for any sequence $\{Q_{n}\}_{n\geq1}\subset\mathcal{M}$ and $Q_{0}%
\in\mathcal{M}$ such that $H_{Q_{n}}$ converges to $H_{Q_{0}}$, $\mu$-a.e.,
then
\[
\rho^{\ast}(Q_{0})\leq\liminf\limits_{n\rightarrow\infty}\rho^{\ast}(Q_{n}).
\]

\end{lemma}

\begin{proof}
Set
\[
L_{+}^{\infty}(\mu)=\{X\in L^{\infty}(\mu):X\geq0\}.
\]
Then $\rho^{\ast}$ can be redefined as
\[
\rho^{\ast}(Q)=\sup\limits_{X\in L_{+}^{\infty}(\mu)}\big(E_{Q}[X]-\rho
(X)\big),
\]
since $E_{Q}[X]-\rho(X)=E_{Q}[X+m]-\rho(X+m)$ for any $Q\in\mathcal{M}$, $X\in
L^{\infty}(\mu)$ and positive real number $m$.

Take $J_{k}=\inf\limits_{n\geq k}H_{Q_{n}}$. Then $\{J_{k}\}_{k\geq1}$ is an
increasing sequence and $H_{Q_{0}}=\sup\limits_{k\geq1}J_{k}$. We have
\[%
\begin{array}
[c]{r@{}l}%
\rho^{\ast}(Q_{0})= & \sup\limits_{X\in L_{+}^{\infty}(\mu)}\big(E_{\mu
}[X(\sup\limits_{k\geq1}J_{k})]-\rho(X)\big)\\
= & \sup\limits_{k\geq1}\sup\limits_{X\in L_{+}^{\infty}(\mu)}\big(E_{\mu
}[XJ_{k}]-\rho(X)\big)\\
= & \sup\limits_{k\geq1}\sup\limits_{X\in L_{+}^{\infty}(\mu)}\big(E_{\mu
}[\inf\limits_{n\geq k}(XH_{Q_{n}})]-\rho(X)\big)\\
\leq & \sup\limits_{k\geq1}\sup\limits_{X\in L_{+}^{\infty}(\mu)}%
\inf\limits_{n\geq k}\big(E_{Q_{n}}[X]-\rho(X)\big)\\
\leq & \sup\limits_{k\geq1}\inf\limits_{n\geq k}\sup\limits_{X\in
L_{+}^{\infty}(\mu)}\big(E_{Q_{n}}[X]-\rho(X)\big)\\
= & \liminf\limits_{n\rightarrow\infty}\rho^{\ast}(Q_{n}).
\end{array}
\]
This completes the proof.
\end{proof}

\begin{lemma}
\label{auxiliary-result} If $\rho_{1}$ and $\rho_{2}$ are convex expectations
continuous from below and Assumption \ref{assumption} holds, then there exists
$Q^{\ast}\in\mathcal{Q}$ such that
\begin{equation}
\inf\limits_{X\in\mathcal{X}_{\alpha}}E_{Q^{\ast}}[K_{2}-X]-\rho_{2}^{\ast
}(Q^{\ast})=\sup\limits_{Q\in\mathcal{Q}}\inf\limits_{X\in\mathcal{X}_{\alpha
}}\big(E_{Q}[K_{2}-X]-\rho_{2}^{\ast}(Q)\big). \label{equation-3.5}%
\end{equation}

\end{lemma}

\begin{proof}
Take a positive constant $0<\epsilon<1$ and a sequence $\{Q_{n}\}_{n\geq
1}\subset\mathcal{Q}$ such that
\[
\inf\limits_{X\in\mathcal{X}_{\alpha}}E_{Q_{n}}[K_{2}-X]-\rho_{2}^{\ast}%
(Q_{n})\geq\gamma-\frac{\epsilon}{2^{n}},
\]
where $\gamma=\sup\limits_{Q\in\mathcal{Q}}\inf\limits_{X\in\mathcal{X}%
_{\alpha}}\big(E_{Q}[K_{2}-X]-\rho_{2}^{\ast}(Q)\big)$. By Lemma
\ref{aux-minimax},
\[
\gamma=\inf\limits_{X\in\mathcal{X}_{\alpha}}\sup\limits_{Q\in\mathcal{Q}%
}\big(E_{Q}[K_{2}-X]-\rho_{2}^{\ast}(Q)\big)=\inf\limits_{X\in\mathcal{X}%
_{\alpha}}\rho_{2}(K_{2}-X).
\]
Since
\[
\rho_{2}(0)\leq\inf\limits_{X\in\mathcal{X}_{\alpha}}\rho_{2}(K_{2}-X),
\]
then $\rho_{2}(0)\leq\gamma$. For any $n$,
\[
M-\rho_{2}^{\ast}(Q_{n})\geq\inf\limits_{X\in\mathcal{X}_{\alpha}}E_{Q_{n}%
}[K_{2}-X]-\rho_{2}^{\ast}(Q_{n})\geq\gamma-\frac{\epsilon}{2^{n}}\geq
\gamma-\epsilon,
\]
which leads to
\[
\rho_{2}^{\ast}(Q_{n})\leq M-\gamma+\epsilon\leq M-\rho_{2}(0)+1.
\]
For $v$ defined in Assumption \ref{assumption}, we have $\rho_{2}^{\ast
}(Q_{n})\leq v$ which implies $\{H_{Q_{n}}\}_{n\geq1}\subset\mathcal{H}_{v}$.

By the Koml\'{o}s Theorem, there exist a subsequence $\{Q_{n_{i}}\}_{i\geq1}$
of $\{Q_{n}\}_{n\geq1}$ and a random variable $H^{\ast}\in L^{1}(\mu)$ such
that
\[
\lim_{k\rightarrow\infty}\frac{1}{k}\sum_{i=1}^{k}H_{Q_{n_{i}}}=H^{\ast}%
,\quad\mu-a.e..
\]
Since $\mathcal{H}_{v}$ is a convex set and closed under the $\mu$-a.e.
convergence, then $H^{\ast}\in\mathcal{H}_{v}$. Denote $Q^{\ast}$ as the
corresponding probability measure of $H^{\ast}$. Since
\[
\lim_{k\rightarrow\infty}\frac{1}{k}\sum_{i=1}^{k}H_{Q_{n_{i}}}=H^{\ast}%
,\quad\mu-a.e.
\]
and
\[
1=E_{\mu}[H^{\ast}]=\lim_{k\rightarrow\infty}E_{\mu}[\frac{1}{k}\sum_{i=1}%
^{k}H_{Q_{n_{i}}}],
\]
we have $\{\frac{1}{k}\sum_{i=1}^{k}H_{Q_{n_{i}}}\}_{k\geq1}$ converges to
$H^{\ast}$ under $L^{1}(\mu)$ norm. By Lemma \ref{Fatou-property-2} and Lemma
\ref{Fatou-property-1},
\[%
\begin{array}
[c]{r@{}l}
& \inf\limits_{X\in\mathcal{X}_{\alpha}}E_{Q^{\ast}}[K_{2}-X]-\rho_{2}^{\ast
}(Q^{\ast})\\
\geq & \limsup\limits_{k\rightarrow\infty}\inf\limits_{X\in\mathcal{X}%
_{\alpha}}E_{\mu}[(K_{2}-X)(\dfrac{1}{k}\sum\limits_{i=1}^{k}H_{Q_{n_{i}}%
})]-\liminf\limits_{k\rightarrow\infty}\rho_{2}^{\ast}(\dfrac{1}{k}%
\sum\limits_{i=1}^{k}Q_{n_{i}})\\
\geq & \limsup\limits_{k\rightarrow\infty}\inf\limits_{X\in\mathcal{X}%
_{\alpha}}\dfrac{1}{k}\sum\limits_{i=1}^{k}\big(E_{Q_{n_{i}}}[(K_{2}%
-X)]-\rho_{2}^{\ast}(Q_{n_{i}})\big)\\
\geq & \limsup\limits_{k\rightarrow\infty}\dfrac{1}{k}\sum\limits_{i=1}%
^{k}\inf\limits_{X\in\mathcal{X}_{\alpha}}\big(E_{Q_{n_{i}}}[(K_{2}%
-X)]-\rho_{2}^{\ast}(Q_{n_{i}})\big)\\
\geq & \lim\limits_{k\rightarrow\infty}(\gamma-\dfrac{\epsilon}{k})=\gamma.
\end{array}
\]
Since $Q^{\ast}\in\mathcal{Q}$, we have
\[
\inf\limits_{X\in\mathcal{X}_{\alpha}}E_{Q^{\ast}}[K_{2}-X]-\rho_{2}^{\ast
}(Q^{\ast})=\sup\limits_{Q\in\mathcal{Q}}\inf\limits_{X\in\mathcal{X}_{\alpha
}}\big(E_{Q}[K_{2}-X]-\rho_{2}^{\ast}(Q)\big).
\]
This completes the proof.
\end{proof}

Summarizing all the lemmas above, we obtain the following proof of Theorem
\ref{minimax-result}:

\begin{proof}
By Lemma \ref{auxiliary-result}, there exists $Q^{\ast}\in\mathcal{Q}$ such
that
\[
\inf\limits_{X\in\mathcal{X}_{\alpha}}E_{Q^{\ast}}[K_{2}-X]-\rho_{2}^{\ast
}(Q^{\ast})=\sup\limits_{Q\in\mathcal{Q}}\inf\limits_{X\in\mathcal{X}_{\alpha
}}\big(E_{Q}[K_{2}-X]-\rho_{2}^{\ast}(Q)\big).
\]
If $X^{\ast}$ is the optimal test of (\ref{initial-problem}), then
\[
\sup\limits_{Q\in\mathcal{Q}}\big(E_{Q}[K_{2}-X^{\ast}]-\rho_{2}^{\ast
}(Q)\big)=\inf\limits_{X\in\mathcal{X}_{\alpha}}\sup\limits_{Q\in\mathcal{Q}%
}\big(E_{Q}[K_{2}-X]-\rho_{2}^{\ast}(Q)\big).
\]
By Lemma \ref{aux-minimax},
\[
\inf\limits_{X\in\mathcal{X}_{\alpha}}\sup\limits_{Q\in\mathcal{Q}}%
\big(E_{Q}[K_{2}-X]-\rho_{2}^{\ast}(Q)\big)=\sup\limits_{Q\in\mathcal{Q}}%
\inf\limits_{X\in\mathcal{X}_{\alpha}}\big(E_{Q}[K_{2}-X]-\rho_{2}^{\ast
}(Q)\big).
\]
Thus,
\[
\inf\limits_{X\in\mathcal{X}_{\alpha}}E_{Q^{\ast}}[K_{2}-X]-\rho_{2}^{\ast
}(Q^{\ast})=\sup\limits_{Q\in\mathcal{Q}}\big(E_{Q}[K_{2}-X^{\ast}]-\rho
_{2}^{\ast}(Q)\big).
\]
Since
\[
\inf\limits_{X\in\mathcal{X}_{\alpha}}E_{Q^{\ast}}[K_{2}-X]-\rho_{2}^{\ast
}(Q^{\ast})\leq E_{Q^{\ast}}[K_{2}-X^{\ast}]-\rho_{2}^{\ast}(Q^{\ast})\leq
\sup\limits_{Q\in\mathcal{Q}}\big(E_{Q}[K_{2}-X^{\ast}]-\rho_{2}^{\ast
}(Q)\big),
\]
then
\[
E_{Q^{\ast}}[K_{2}-X^{\ast}]-\rho_{2}^{\ast}(Q^{\ast})=\inf\limits_{X\in
\mathcal{X}_{\alpha}}E_{Q^{\ast}}[K_{2}-X]-\rho_{2}^{\ast}(Q^{\ast}),
\]
i.e.,
\[
E_{Q^{\ast}}[K_{2}-X^{\ast}]=\inf\limits_{X\in\mathcal{X}_{\alpha}}E_{Q^{\ast
}}[K_{2}-X].
\]
This completes the proof.
\end{proof}

\begin{example}
\label{example-helpless} Consider the probability space $(\Omega
,\mathcal{F},\mu)$, where $\Omega=\{0,1\}$, $\mathcal{F}=\{\emptyset
,\{0\},\{1\},\Omega\}$. Set
\[
\mu(\omega)=\Bigg\{%
\begin{array}
[c]{l@{}c}%
\frac{1}{2}, & \quad\text{if }\omega=0,\\
\frac{1}{2}, & \quad\text{if }\omega=1,
\end{array}
\text{\ \ and\ \ }Q_{0}(\omega)=\Bigg\{%
\begin{array}
[c]{l@{}c}%
\frac{3}{4}, & \quad\text{if }\omega=0,\\
\frac{1}{4}, & \quad\text{if }\omega=1
\end{array}
.
\]
Let $K_{1}=0$, $K_{2}=1$, $\alpha=\frac{1}{2}$, $\rho_{1}(X)=E_{\mu}[X]$ and
$\rho_{2}(X)=\ln E_{Q_{0}}[e^{X}]$. We solve the problem
(\ref{initial-problem}). Let $Q=qI_{\{0\}}+(1-q)I_{\{1\}}$, where $0\leq
q\leq1$. Then
\[
\rho_{2}^{\ast}(Q)=E_{Q_{0}}[\frac{dQ}{dQ_{0}}\ln\frac{dQ}{dQ_{0}}]=q\ln
q+(1-q)\ln(1-q)-q\ln3+2\ln2.
\]
Let $X=x_{0}I_{\{0\}}+x_{1}I_{\{1\}}$, where $0\leq x_{0},$ $x_{1}\leq1$. If
$X\in\mathcal{X}_{\alpha}$, then $\frac{1}{2}x_{0}+\frac{1}{2}x_{1}\leq
\frac{1}{2}$, i.e., $x_{0}\leq1-x_{1}$. When $q=\frac{3}{e+3}$, $\inf
\limits_{X\in\mathcal{X}_{\alpha}}E_{Q}[1-X]-\rho_{2}^{\ast}(Q)$ attains its
maximum. Thus,
\[
Q^{\ast}=\frac{3}{e+3}I_{\{0\}}+\frac{e}{e+3}I_{\{1\}}\text{ and }X^{\ast
}=I_{\{0\}}.
\]

\end{example}

\subsection{The existence of a representative probability $P^{\ast}$}

In the rest of this paper, $Q^{\ast}$ is always\ the probability measure found
in Theorem \ref{minimax-result}. Define
\[
\gamma_{\alpha}=\inf\limits_{X\in\mathcal{X}_{\alpha}}E_{Q^{\ast}}[K_{2}-X].
\]
If $\gamma_{\alpha}=0$, then it is trivial and the optimal test $X^{\ast
}=K_{2}$, $Q^{\ast}$-a.e.. In the following, we only consider the case
$\gamma_{\alpha}>0$.

\begin{lemma}
\label{dual-problem} If $\gamma_{\alpha}>0$, $\rho_{1}$ and $\rho_{2}$ are
convex expectations continuous from below and Assumption \ref{assumption}
holds, then for any optimal test $X^{\ast}$ of (\ref{initial-problem}), we
have $X^{\ast}\in\mathcal{X}^{\gamma_{\alpha}}$ and
\begin{equation}
\rho_{1}(X^{\ast})=\inf\limits_{X\in\mathcal{X}^{\gamma_{\alpha}}}\rho
_{1}(X)=\alpha, \label{2nd-dual-problem}%
\end{equation}
where $\mathcal{X}^{\gamma_{\alpha}}=\{X:E_{Q^{\ast}}[K_{2}-X]\leq
\gamma_{\alpha},K_{1}\leq X\leq K_{2},X\in L^{\infty}(\mu)\}$.
\end{lemma}

\begin{proof}
$X^{\ast}\in\mathcal{X}^{\gamma_{\alpha}}$ comes from Theorem
\ref{minimax-result}. For any $X\in\mathcal{X}_{\alpha}$, if $\rho
_{1}(X)<\alpha$, we claim $E_{Q^{\ast}}[K_{2}-X]>\gamma_{\alpha}$. If not,
then there will exist a test $X^{\prime}\in\mathcal{X}_{\alpha}$ such that
$\rho_{1}(X^{\prime})<\alpha$ and
\[
E_{Q^{\ast}}[K_{2}-X^{\prime}]=\gamma_{\alpha}.
\]
Set
\[
\rho_{1}(X^{\prime})=\alpha^{\prime}<\alpha
\]
and
\[
X^{\prime\prime}=(X^{\prime}+\alpha-\alpha^{\prime})\wedge K_{2}.
\]
By the definition of convex expectation,
\[
\rho_{1}(X^{\prime\prime})\leq\rho_{1}(X^{\prime}+\alpha-\alpha^{\prime}%
)=\rho_{1}(X^{\prime})+\alpha-\alpha^{\prime}=\alpha,
\]
which implies that $X^{\prime\prime}\in\mathcal{X}_{\alpha}$. As
$X^{\prime\prime}\in\mathcal{X}_{\alpha}$ and $X^{\prime\prime}\geq X^{\prime
}$, we have $E_{Q^{\ast}}[K_{2}-X^{\prime\prime}]=E_{Q^{\ast}}[K_{2}-X^{\prime
}]$, i.e., $E_{Q^{\ast}}[X^{\prime\prime}]=E_{Q^{\ast}}[X^{\prime}]$. Set
$A=\{X^{\prime}\not =K_{2}\}$. Since
\[
X^{\prime\prime}-X^{\prime}\geq0\quad\text{and}\quad E_{Q^{\ast}}%
[X^{\prime\prime}-X^{\prime}]\geq0,
\]
we have $X^{\prime\prime}=X^{\prime}$, $Q^{\ast}$-a.e., which implies that
$Q^{\ast}(A)=0$ and $X^{\prime}=K_{2}$, $Q^{\ast}$-a.e.. Then $\gamma_{\alpha
}=0$, which contradicts with $\gamma_{\alpha}>0$.

Thus, for any $X\in\mathcal{X}^{\gamma_{\alpha}}$, we have $\rho_{1}%
(X)\geq\alpha$. With $\rho_{1}(X^{\ast})=\alpha$, the result holds.
\end{proof}

\begin{theorem}
\label{Second-minimax-result} Suppose that $\gamma_{\alpha}>0$, $\rho_{1}$ and
$\rho_{2}$ are convex expectations continuous from below and Assumption
\ref{assumption} holds. Then there exists $P^{\ast}\in\mathcal{P}$ such that
for any optimal test $X^{\ast}$ of (\ref{initial-problem}),
\[
E_{P^{\ast}}[X^{\ast}]=\inf\limits_{X\in\mathcal{X}^{\gamma_{\alpha}}%
}E_{P^{\ast}}[X].
\]

\end{theorem}

\begin{proof}
Set $Y=K_{2}-X$, $Y^{\ast}=K_{2}-X^{\ast}$ and
\[
\mathcal{Y}_{\gamma_{\alpha}}=\{Y:E_{Q^{\ast}}[Y]\leq\gamma_{\alpha},0\leq
Y\leq K_{2}-K_{1},Y\in L^{\infty}(\mu)\}.
\]
By Lemma \ref{dual-problem},
\[
\rho_{1}(K_{2}-Y^{\ast})=\inf\limits_{Y\in\mathcal{Y}_{\gamma_{\alpha}}}%
\rho_{1}(K_{2}-Y),
\]
i.e.,
\begin{equation}
\sup\limits_{P\in\mathcal{P}}\big(E_{P}[K_{2}-Y^{\ast}]-\rho_{1}^{\ast
}(P)\big)=\inf\limits_{Y\in\mathcal{Y}_{\gamma_{\alpha}}}\sup\limits_{P\in
\mathcal{P}}\big(E_{P}[K_{2}-Y]-\rho_{1}^{\ast}(P)\big). \label{3.4}%
\end{equation}
Applying similar analysis as in Lemma \ref{weak*-compact}, we obtain that
$\mathcal{Y}_{\gamma_{\alpha}}$ is compact in the topology $\sigma(L^{\infty
}(\mu),L^{1}(\mu))$. By the minimax theorem,
\begin{equation}
\inf\limits_{Y\in\mathcal{Y}_{\gamma_{\alpha}}}\sup\limits_{P\in\mathcal{P}%
}\big(E_{P}[K_{2}-Y]-\rho_{1}^{\ast}(P)\big)=\sup\limits_{P\in\mathcal{P}}%
\inf\limits_{Y\in\mathcal{Y}_{\gamma_{\alpha}}}\big(E_{P}[K_{2}-Y]-\rho
_{1}^{\ast}(P)\big). \label{3.5}%
\end{equation}

Now we prove that there exists a probability measure $P^{\ast}\in\mathcal{P}$
such that
\begin{equation}
\inf\limits_{Y\in\mathcal{Y}_{\gamma_{\alpha}}}\big(E_{P^{\ast}}[K_{2}%
-Y]-\rho_{1}^{\ast}(P^{\ast})\big)=\sup\limits_{P\in\mathcal{P}}%
\inf\limits_{Y\in\mathcal{Y}_{\gamma_{\alpha}}}\big(E_{P}[K_{2}-Y]-\rho
_{1}^{\ast}(P)\big). \label{3.6}%
\end{equation}
If we replace $X$ by $Y$, $\mathcal{X}_{\alpha}$ by $\mathcal{Y}%
_{\gamma_{\alpha}}$, $P$ by $Q$ and $\rho_{1}^{\ast}$ by $\rho_{2}^{\ast}$ in
(\ref{equation-3.5}), then (\ref{equation-3.5}) becomes (\ref{3.6}). Using the
same proof method as in Lemma \ref{auxiliary-result}, we deduce that (\ref{3.6}) holds.

By (\ref{3.4}), (\ref{3.5}) and (\ref{3.6}),
\[
\inf\limits_{Y\in\mathcal{Y}_{\gamma_{\alpha}}}\big(E_{P^{\ast}}[K_{2}%
-Y]-\rho_{1}^{\ast}(P^{\ast})\big)=\sup\limits_{P\in\mathcal{P}}%
\big(E_{P}[K_{2}-Y^{\ast}]-\rho_{1}^{\ast}(P)\big).
\]
Since
\[%
\begin{array}
[c]{r@{}l}%
\inf\limits_{Y\in\mathcal{Y}_{\gamma_{\alpha}}}E_{P^{\ast}}[K_{2}-Y]-\rho
_{1}^{\ast}(P^{\ast})\leq & E_{P^{\ast}}[K_{2}-Y^{\ast}]-\rho_{1}^{\ast
}(P^{\ast})\\
\leq & \sup\limits_{P\in\mathcal{P}}\big(E_{P}[K_{2}-Y^{\ast}]-\rho_{1}^{\ast
}(P)\big),
\end{array}
\]
we have
\[
E_{P^{\ast}}[K_{2}-Y^{\ast}]-\rho_{1}^{\ast}(P^{\ast})=\inf\limits_{Y\in
\mathcal{Y}_{\gamma_{\alpha}}}E_{P^{\ast}}[K_{2}-Y]-\rho_{1}^{\ast}(P^{\ast
}).
\]
Thus,
\[
E_{P^{\ast}}[K_{2}-Y^{\ast}]=\inf\limits_{Y\in\mathcal{Y}_{\gamma_{\alpha}}%
}E_{P^{\ast}}[K_{2}-Y],
\]
i.e.,
\[
E_{P^{\ast}}[X^{\ast}]=\inf\limits_{X\in\mathcal{X}^{\gamma_{\alpha}}%
}E_{P^{\ast}}[X].
\]
This completes the proof.
\end{proof}

\begin{example}
\label{example-discrete} Consider the probability space $(\Omega
,\mathcal{F},\mu)$, where $\Omega$, $\mathcal{F}$ and $\mu$ are defined as the
same as in Example \ref{example-helpless}. Set $K_{1}=0$, $K_{2}=1$,
$\alpha=\ln(e+3)-2\ln2$, $\rho_{1}(X)=\ln E_{P_{0}}[e^{X}]$ and $\rho
_{2}(X)=E_{\mu}[X]$, where
\[
P_{0}(\omega)=\Bigg\{%
\begin{array}
[c]{l@{}c}%
\frac{1}{4}, & \quad\text{if }\omega=0,\\
\frac{3}{4}, & \quad\text{if }\omega=1.
\end{array}
\]
We solve the problem (\ref{initial-problem}). It is easy to check that
\[
\inf\limits_{X\in\mathcal{X}_{\alpha}}E_{\mu}(1-X)=\frac{1}{2},
\]
i.e., $\gamma_{\alpha}=\frac{1}{2}$. By Lemma \ref{dual-problem}, to solve the
problem (\ref{initial-problem}) is equivalent to solve the following problem:
\begin{equation}
\text{minimize}\quad\rho_{1}(X), \label{problem-initial-extend}%
\end{equation}
over the set $\mathcal{X}^{\gamma_{\alpha}}=\{X:E_{\mu}[X]\geq\frac{1}%
{2},0\leq X\leq1\}$. Let $X=x_{0}I_{\{0\}}+x_{1}I_{\{1\}}$, where $0\leq
x_{0},x_{1}\leq1$. If $X\in\mathcal{X}^{\gamma_{\alpha}}$, then $x_{0}%
\geq1-x_{1}$. Let $P=pI_{\{0\}}+(1-p)I_{\{1\}}$, where $0\leq p\leq1$. Then
\[
\rho_{1}^{\ast}(P)=E_{P_{0}}[\frac{dP}{dP_{0}}\ln\frac{dP}{dP_{0}}]=2\ln2+p\ln
p+(1-p)\ln(1-p)-(1-p)\ln3.
\]
When $p=\frac{e}{e+3}$, $\inf\limits_{X\in\mathcal{X}^{\gamma_{\alpha}}}%
E_{P}[X]-\rho_{1}^{\ast}(P)$ attains its maximum. Thus,
\[
P^{\ast}=\frac{e}{e+3}I_{\{0\}}+\frac{3}{e+3}I_{\{1\}}\text{ \ and \ }X^{\ast
}=I_{\{0\}}.
\]

\end{example}

\subsection{Main result\label{main results}}

\begin{theorem}
\label{main-result} If $\rho_{1}$ and $\rho_{2}$ are convex expectations
continuous from below and Assumption \ref{assumption} holds, then there exist
$P^{\ast}\in\mathcal{P}$ and $Q^{\ast}\in\mathcal{Q}$ such that for any
optimal test $X^{\ast}$ of (\ref{initial-problem}), it can be expressed as
\begin{equation}
X^{\ast}=K_{2}I_{\{H_{Q^{\ast}}>zG_{P^{\ast}}\}}+BI_{\{H_{Q^{\ast}%
}=zG_{P^{\ast}}\}}+K_{1}I_{\{H_{Q^{\ast}}<zG_{P^{\ast}}\}},\quad\mu-a.e.,
\label{form}%
\end{equation}
where $z\in\lbrack0,+\infty)\cup\{+\infty\}$ and $B$ is a random variable
taking values in the random interval $[K_{1},K_{2}]$.
\end{theorem}

\begin{proof}
We divide our proof into two cases:

i) The case $\gamma_{\alpha}>0$. By Theorem \ref{Second-minimax-result},
$X^{\ast}$ is the optimal test of the following problem:
\[
\text{minimize}\quad E_{P^{\ast}}[X],
\]
over the set $\mathcal{X}^{\gamma_{\alpha}}=\{X:E_{Q^{\ast}}[K_{2}%
-X]\leq\gamma_{\alpha},K_{1}\leq X\leq K_{2},X\in L^{\infty}(\mu)\}$. Set
\[
Z^{\ast}=\frac{K_{2}-X^{\ast}}{K_{2}-K_{1}},\, Z=\frac{K_{2}-X}{K_{2}-K_{1}%
},\,\gamma_{\alpha}^{\prime}=\frac{\gamma_{\alpha}}{E_{Q^{\ast}}[K_{2}-K_{1}%
]},\,\frac{d\hat{P}}{dP^{\ast}}=\frac{K_{2}-K_{1}}{E_{P^{\ast}}[K_{2}-K_{1}%
]}\,\text{and}\,\frac{d\hat{Q}}{dQ^{\ast}}=\frac{K_{2}-K_{1}}{E_{Q^{\ast}%
}[K_{2}-K_{1}]}.
\]
Then $Z^{\ast}$ is the optimal test of the problem:
\begin{equation}
\text{maximize}\quad E_{\hat{P}}[Z], \label{linear-form}%
\end{equation}
over the set $\mathcal{Z}_{\gamma_{\alpha}^{\prime}}=\{Z:E_{\hat{Q}}%
[Z]\leq\gamma_{\alpha}^{\prime},0\leq Z\leq1,Z\in L^{\infty}(\mu)\}$.

By the classical Neyman-Pearson lemma (see \cite{r1} or Theorem A.30 in
\cite{r12}), any optimal test $Z^{\ast}$ of (\ref{linear-form}) has the form
\begin{equation}
Z^{\ast}=I_{\{z^{\prime}H_{\hat{Q}}<G_{\hat{P}}\}}+B^{\prime}\cdot
I_{\{z^{\prime}H_{\hat{Q}}=G_{\hat{P}}\}},\quad\mu-a.e.
\end{equation}
for some constant $z^{\prime}\geq0$ and random variable $B^{\prime}$ taking
values in the interval $[0,1]$. Since
\[
\frac{d\hat{P}}{dP^{\ast}}=\frac{K_{2}-K_{1}}{E_{P^{\ast}}[K_{2}-K_{1}]}%
\quad\text{and}\quad\frac{d\hat{Q}}{dQ^{\ast}}=\frac{K_{2}-K_{1}}{E_{Q^{\ast}%
}[K_{2}-K_{1}]},
\]
if we take (with conventions $+\infty=\frac{1}{0}$ and $0=\frac{0}{0}$)
\[
B=K_{2}-(K_{2}-K_{1})B^{\prime}\quad\text{and}\quad z=\frac{E_{Q^{\ast}}%
[K_{2}-K_{1}]}{z^{\prime}E_{P^{\ast}}[K_{2}-K_{1}]},
\]
then $z^{\prime}\in(0, +\infty)\cup\{+\infty\}$ and
\begin{equation}
\label{representation-form}X^{\ast}=K_{2}I_{\{H_{Q^{\ast}}>zG_{P^{\ast}}%
\}}+BI_{\{H_{Q^{\ast}}=zG_{P^{\ast}}\}}+K_{1}I_{\{H_{Q^{\ast}}<zG_{P^{\ast}%
}\}},\quad\mu-a.e..
\end{equation}

ii) The case $\gamma_{\alpha}=0$. For this case, $X^{\ast}=K_{2}$, $Q^{\ast}%
$-a.e.. This is a special case of (\ref{representation-form}) when $z$ equals
$0$.
\end{proof}

\begin{example}
\label{main-example} Except $\rho_{2}(X)=\ln E_{Q_{0}}[e^{X}]$ where $Q_{0}$
is defined as in Example \ref{example-helpless}, all the notations in this
example are defined as the same as in Example \ref{example-discrete}. We solve
the problem (\ref{initial-problem}).

Denote $\mathcal{Z}=\{X:0\leq X\leq1,E_{\mu}[X]\leq\frac{1}{2}\}$. By Example
\ref{example-discrete}, we have $\sup\limits_{X\in\mathcal{X}_{\alpha}}E_{\mu
}[X]=\frac{1}{2}$. Then $\mathcal{X}_{\alpha}\subset\mathcal{Z}$ and
\begin{equation}
\inf\limits_{X\in\mathcal{Z}}\rho_{2}(1-X)\leq\inf\limits_{X\in\mathcal{X}%
_{\alpha}}\rho_{2}(1-X). \label{au-example}%
\end{equation}
Take $\hat{X}=I_{\{0\}}$. By Example \ref{example-helpless},
\[
\rho_{2}(1-\hat{X})=\inf\limits_{X\in\mathcal{Z}}\rho_{2}(1-X).
\]
Since $\hat{X}\in\mathcal{X}_{\alpha}$, with (\ref{au-example}), we have
\[
\rho_{2}(1-\hat{X})=\inf\limits_{X\in\mathcal{X}_{\alpha}}\rho_{2}(1-X),
\]
which implies $I_{\{0\}}$ is the optimal test. Furthermore, if we take
$Q^{\ast}=\frac{3}{e+3}I_{\{0\}}+\frac{e}{e+3}I_{\{1\}}$ and $P^{\ast}%
=\frac{e}{e+3}I_{\{0\}}+\frac{3}{e+3}I_{\{1\}}$ as in Examples
\ref{example-helpless} and \ref{example-discrete}, then
\[
I_{\{0\}}=I_{\{\frac{3}{e}H_{Q^{\ast}}>G_{P^{\ast}}\}}.
\]

\end{example}

\section{Application}

In a financial market, if an investor does not have enough initial wealth,
then he may fail to (super-) hedge an contingent claim and will face some
shortfall risk. In this case, we need a criterion expressing the investor's
attitude towards the shortfall risk (see \cite{fo-le-1999,fo-le-2000,
fo-2002-b, r6}). F\"{o}llmer and Leukert \cite{fo-le-2000} use the expectation
of the shortfall weighted by the loss function as a shortfall risk measure. In
this section, we use a general measure, the convex risk measure, to evaluate
the shortfall and consequently minimize such a shortfall risk.

In more details, we adopt the same financial market model as in
\cite{fo-le-2000}. The discounted price process of the underlying asset is
described as a semimartingale $S=(S_{t})_{t\in\lbrack0,T]}$ on a complete
probability space $(\Omega,\mathcal{F},\mu)$. The information structure is
given by a filtration $F=\{\mathcal{F}_{t}\}_{0\leq t\leq T}$ with
$\mathcal{F}_{T}=\mathcal{F}$. Let $\mathcal{P}$ denote the set of equivalent
martingale measures. we assume that $\mathcal{F}_{0}$ is trivial and
$\mathcal{P\neq\emptyset}$. For an initial investment $X_{0}\geq0$ and a
portfolio process $\pi$ such that the wealth process%
\begin{equation}
X_{t}=X_{0}+\int_{0}^{t}\pi_{s}dS_{s}\;\;\forall t\in\lbrack0,T]
\label{wealth equation}%
\end{equation}
is well defined. A strategy $(X_{0},\pi)$ is called admissible if the
corresponding wealth process $X$ is nonnegative. For a given nonnegative
contingent claim $H\in L^{\infty}(\mu)$, we define that
\[
U_{0}=\underset{P\in\mathcal{P}}{\sup}E_{P}[H].
\]
It is well known that if the investor's initial wealth $\tilde{X}_{0}<U_{0}$,
then some shortfall $(H-X_{T})^{+}$ will occur at time $T$.

In this section, we introduce a general convex expectation $\rho$ to measure
the shortfall $(H-X_{T})^{+}$.

\begin{definition}
For a given convex expectation $\rho$, the shortfall risk is defined as
\[
\rho((H-X_{T})^{+})\text{.}%
\]

\end{definition}

Consequently, the investor wants to find an admissible strategy $(X_{0},\pi)$
which minimizes the shortfall risk and control his initial investment
$X_{0}\leq\tilde{X}_{0}$. Thus, we will solve the following optimization
problem:%
\begin{equation}%
\begin{array}
[c]{c}%
\underset{(X_{0},\pi)}{\text{min}}\;\rho((H-X_{T})^{+}),\\
\text{subject to }X_{0}\leq\tilde{X}_{0},%
\end{array}
\label{convex risk optimization-1}%
\end{equation}
where $\tilde{X}_{0}$ is the initial wealth of the investor.

Now we show that the optimal $X_{T}^{\ast}$ must satisfy $0\leq X_{T}^{\ast
}\leq H$. In fact, if $P(X_{T}^{\ast}>H)>0$, we can construct a feasible
terminal wealth $\tilde{X}_{T}$ such that $0\leq\tilde{X}_{T}\leq H$ and
$(H-\tilde{X}_{T})^{+}<(H-X_{T}^{\ast})^{+}$. Thus, $\rho((H-\tilde{X}%
_{T})^{+})<\rho((H-X_{T}^{\ast})^{+})$ by the monotonicity property of $\rho$.
This leads to a contradiction.

Thus, without loss of generality we assume that $0\leq X_{T}\leq H$ and
(\ref{convex risk optimization-1}) becomes%
\begin{equation}%
\begin{array}
[c]{c}%
\underset{0\leq X_{T}\leq H}{\text{min}}\rho(H-X_{T}),\\
\text{subject to }\underset{P\in\mathcal{P}}{\sup}E_{P}[X_{T}]\leq\tilde
{X}_{0}.
\end{array}
\label{convex risk optimization-2}%
\end{equation}

By Theorem \ref{main-result} and the classical Neyman-Pearson lemma
(Proposition 4.1 in \cite{fo-le-2000}),%
\[
X_{T}^{\ast}=HI_{\{zH_{Q^{\ast}}>G_{P^{\ast}}\}}+BI_{\{zH_{Q^{\ast}%
}=G_{P^{\ast}}\}},\quad\mu-a.e.,
\]
where%
\[
z=\sup\{\tilde{z}\mid\int_{\{\tilde{z}H_{Q^{\ast}}>G_{P^{\ast}}\}}HdP^{\ast
}\leq\tilde{X}_{0}\}
\]
and%
\[
B=\left\{
\begin{array}
[c]{ll}%
\frac{\tilde{X}_{0}-\int_{\{zH_{Q^{\ast}}>G_{P^{\ast}}\}}HdP^{\ast}}%
{\int_{\{zH_{Q^{\ast}}=G_{P^{\ast}}\}}HdP^{\ast}}, & \;\;\text{when }P^{\ast
}[\{H>0\}\cap\{zH_{Q^{\ast}}=G_{P^{\ast}}\}]>0;\\
0, & \;\;\text{otherwise.}%
\end{array}
\right.
\]

Then by the optional decomposition theorem (see \cite{kra} and \cite{fo-kab}),
we obtain the optimal strategy $(\tilde{X}_{0},\pi^{\ast})$ corresponding to
$X_{T}^{\ast}$.

\begin{remark}
Instead of minimizing the convex risk measure under the initial investment
constraint, we can solve the following essentially equivalent problem: fix a
bound on the convex risk measure and minimize the initial investment.
\end{remark}

\begin{remark}
We assume that the given nonnegative contingent claim $H\in L^{\infty}(\mu)$.
If $H\in L^{1}(\mu)$, then we can use Theorem \ref{auxiliary-theorem} in the appendix.
\end{remark}

\section{Appendix}

In this appendix, we first prove that when the convex expectations are
continuous from above, Assumption \ref{assumption} holds naturally. Then an
example is given to show that Assumption (\ref{assumption}) is only a
sufficient condition for the existence of $Q^{\ast}$. Finally, we give the
Neyman-Pearson lemma for convex expectations on $L^{1}(\mu)$.

\begin{definition}
\label{continuous from above} We call a convex expectation $\rho$ is
continuous from above iff for any sequence $\{X_{n}\}_{n\geq1}\subset
L^{\infty}(\mu)$ decreases to some $X\in L^{\infty}(\mu)$, then $\rho
(X_{n})\to\rho(X)$.
\end{definition}

\begin{proposition}
\label{a-result-from above} If $\rho_{1}$ and $\rho_{2}$ are continuous from
above, then Assumption \ref{assumption} holds.
\end{proposition}

\begin{proof}
We only show the result holds for $\rho_{1}$.

For any $u>\max\{0,M-\rho_{1}(0)+1\}$, we have $u>\max\{0,-\rho_{1}(0)\}$. By
Theorem 3.6 in \cite{r13}, $\mathcal{G}_{u}$ is uniformly integrable. For any
sequence $\{G_{P_{n}}\}_{n\geq1}\subset\mathcal{G}_{u}$ that converges to
$G_{\hat{P}}$, $\mu$-a.e., since $\{G_{P_{n}}\}_{n\geq1}$ is uniformly
integrable,%
\[
E_{\mu}[G_{\hat{P}}]=\lim\limits_{n\rightarrow\infty}E_{\mu}[G_{P_{n}}]=1,
\]
which shows $\hat{P}\in\mathcal{M}$. On the other hand, for any $u>\max
\{0,M-\rho_{1}(0)+1\}$, by Lemma \ref{Fatou-property-1}, we have
\[
\rho^{\ast}(\hat{P})\leq\liminf\limits_{n\rightarrow\infty}\rho^{\ast}%
(P_{n})\leq u.
\]
Then $G_{\hat{P}}\in\mathcal{G}_{u}$. Thus, $\mathcal{G}_{u}$ is closed under
the $\mu$-a.e. convergence.
\end{proof}

Now we show that even if Assumption (\ref{assumption}) does not hold, the
probability measure $Q^{\ast}$ may still exist.

\begin{example}
\label{example-continuous} Consider the probability space $(\Omega
,\mathcal{B},\mu)$, where $\Omega$ is the interval $[0,1]$, $\mathcal{B}$ is
the collection of all Borel sets in $[0,1]$ and $\mu$ is the Lebesgue measure
defined on $[0,1]$. Set $K_{1}=0$, $K_{2}=1$, $\alpha=\frac{3-e}{e-1}$,
$\rho_{1}(X)=E_{P}[X]$ and $\rho_{2}(X)=\ln E_{\mu}[e^{X}]$, where
\[
\frac{dP}{d\mu}=\Bigg\{%
\begin{array}
[c]{l@{}c}%
\frac{e+1}{e-1}, & \quad\omega\in\lbrack0,\frac{e-2}{e-1}],\\
\frac{3-e}{e-1}, & \quad\omega\in(\frac{e-2}{e-1},1].
\end{array}
\]
To solve the problem (\ref{initial-problem}), one can check that\ Assumption
\ref{assumption} does not hold. Let
\[
X^{\ast}=I_{(\frac{e-2}{e-1},1]}\text{ \ and\ }\frac{dQ^{\ast}}{d\mu}=\Bigg\{%
\begin{array}
[c]{l@{}c}%
\frac{e}{e-1}, & \quad\omega\in\lbrack0,\frac{e-2}{e-1}],\\
\frac{1}{e-1}, & \quad\omega\in(\frac{e-2}{e-1},1].
\end{array}
\]
We will show that $X^{\ast}$ is the optimal test and $Q^{\ast}$ satisfies
\begin{equation}
\sup\limits_{X\in\mathcal{X}_{\alpha}}E_{Q^{\ast}}[X]+\rho_{2}^{\ast}(Q^{\ast
})=\inf\limits_{Q\in\mathcal{Q}}\sup\limits_{X\in\mathcal{X}_{\alpha}%
}\big(E_{Q}[X]+\rho_{2}^{\ast}(Q)\big). \label{condition satisfied by Q}%
\end{equation}
In fact, through simple calculations, we obtain%
\[
E_{Q^{\ast}}[1-X^{\ast}]-\rho_{2}^{\ast}(Q^{\ast})=\rho_{2}(1-X^{\ast}).
\]
Otherwise, by the classical Neyman-Pearson lemma, we know that $X^{\ast}$ is
also the optimal test for discriminating between probability measures $P$ and
$Q^{\ast}$, i.e.,%
\[
E_{Q^{\ast}}[1-X^{\ast}]=\inf_{X\in\mathcal{X}_{\alpha}}E_{Q^{\ast}}[1-X].
\]
Since
\[
\inf_{X\in\mathcal{X}_{\alpha}}\rho_{2}(1-X)\geq\inf_{X\in\mathcal{X}_{\alpha
}}E_{Q^{\ast}}[1-X]-\rho_{2}^{\ast}(Q^{\ast})=E_{Q^{\ast}}[1-X^{\ast}%
]-\rho_{2}^{\ast}(Q^{\ast})
\]
and
\[
E_{Q^{\ast}}[1-X^{\ast}]-\rho_{2}^{\ast}(Q^{\ast})=\rho_{2}(1-X^{\ast}%
)\geq\inf_{X\in\mathcal{X}_{\alpha}}\rho_{2}(1-X),
\]
we have
\[
\inf_{X\in\mathcal{X}_{\alpha}}E_{Q^{\ast}}[1-X]-\rho_{2}^{\ast}(Q^{\ast
})=\inf_{X\in\mathcal{X}_{\alpha}}\rho_{2}(1-X),
\]
which leads to (\ref{condition satisfied by Q}).
\end{example}

In the next, we consider the case that $\rho_{1}$ and $\rho_{2}$ are two
convex expectations defined on $L^{1}(\mu)$. Then our problem becomes:
\begin{equation}
\text{minimize}\quad\rho_{2}(K_{2}-X), \label{L-1-problem}%
\end{equation}
over the set $\mathcal{X}_{\alpha}=\{X:K_{1}\leq X\leq K_{2},\rho_{1}%
(X)\leq\alpha,X\in L^{1}(\mu)\}$, where $K_{1},K_{2}\in L^{1}(\mu)$.

\begin{theorem}
\label{auxiliary-theorem} If $\rho_{1}$ and $\rho_{2}$ are two finite convex
expectations defined on $L^{1}(\mu)$ space, then the optimal test of
(\ref{L-1-problem}) exists and has the same form as in Theorem
\ref{main-result}.
\end{theorem}

\begin{proof}
Since $\rho_{1}$ and $\rho_{2}$ are finite, then they are Lebesgue-continuous.
Repeating the proof of Theorem {\ref{existence}}, we will get the optimal test
exists. On the other hand, since $\rho_{1}$ and $\rho_{2}$ can be represented
by some set $\mathcal{P}$ and $\mathcal{Q}$ with their densities sets
$\{G_{P}\in L^{\infty}(\mu):P\in\mathcal{P}\}$ and $\{H_{Q}\in L^{\infty}%
(\mu):Q\in\mathcal{Q}\}$ are weakly compact, the form in Theorem
\ref{main-result} can also be obtained by using the same method as in section
4. The detailed proof is omitted.
\end{proof}

\end{document}